\documentclass[11pt]{article}
\usepackage{amssymb,amsmath}
\usepackage{cases}
\usepackage{amsthm}
\usepackage{amsfonts}
\usepackage{makecell}
\usepackage{cite,color,xcolor}
\usepackage[margin=1.4 in]{geometry}
\usepackage[colorlinks,citecolor=blue,urlcolor=blue]{hyperref}
\usepackage{tikz}

\usetikzlibrary{patterns}
\usepackage[utf8]{inputenc}
\usepackage[pagewise]{lineno}
\usepackage{pgfplots}

\newtheorem{theorem}{Theorem}[section]
\newtheorem{corollary}{Corollary}[section]
\newtheorem{lemma}{Lemma}[section]

\newtheorem{remark}{Remark}[section]

\usepackage{txfonts}
\newcommand{\bal}{\begin{align}}
\newcommand{\bbal}{\begin{align*}}
\newcommand{\beq}{\begin{equation}}
\newcommand{\eeq}{\end{equation}}
\newcommand{\bca}{\begin{cases}}
\newcommand{\eca}{\end{cases}}

\newcommand{\pa}{\partial}
\newcommand{\fr}{\frac}
\newcommand{\na}{\nabla}

\newcommand{\ep}{\varepsilon}
\newcommand{\dd}{\mathrm{d}}

\newcommand{\R}{\mathbb{R}}

\newcommand{\les}{\lesssim}

\newcommand{\f}{\left}
\newcommand{\g}{\right}

\numberwithin{equation}{section}

\begin{document}
\title{Norm inflation and low-regularity ill-posedness for the rod equation}

\author{
 Jinlu Li
 \footnote{
School of Mathematics and Computer Sciences, Gannan Normal University, Ganzhou 341000, China.
\text{E-mail: lijinlu@gnnu.edu.cn}}
\quad and\quad
Yanghai Yu
\footnote{
School of Mathematics and Statistics, Anhui Normal University, Wuhu 241002, China.
\text{E-mail: yuyanghai214@sina.com} (Corresponding author)}
}
\date{\today}
\maketitle
\begin{abstract}
In this paper, we consider the Cauchy problem for the rod equation in the line.  By constructing an explicit smooth initial data, we present a new method to prove that this problem is ill-posed in $H^s(\R)$ with $1< s<3/2$ in the sense of {\it norm inflation}, i.e., an initial data is smooth and arbitrarily small in $H^s(\R)$ with $1< s<3/2$, but the solution becomes arbitrarily large in the Sobolev space after an arbitrarily short time.
\end{abstract}



\section{Introduction}

\bibliographystyle{plain}
\subsection{Models and Background}\label{subsec1}
We now discuss some background on the compressible hyper-elastic rod equation and review some of the history and previous work on this problem. Although a rod is always three-dimensional, if its diameter is much less than the axial length scale, one-dimensional equations can give a good description of the motion of the rod.  The propagation of nonlinear waves inside cylindrical hyper-elastic rods, assuming that the diameter is small when compared to the axial length scale, is described by the one dimensional equation which was derived by Dai \cite{dai1} as a new (one-dimensional) nonlinear dispersive equation including extra nonlinear terms involving second-order and third-order derivatives for a compressible hyper-elastic material
\begin{align}\label{or}
v_\tau+\sigma_1 v v_{\xi}+\sigma_2 v_{\xi \xi \tau}+\sigma_3\left(2 v_{\xi} v_{\xi \xi}+v v_{\xi \xi \xi}\right)=0, \quad \xi \in \mathbb{R},\; \tau>0,
\end{align}
where $v(\tau, \xi)$ represents the radial stretch relative to a pre-stressed state, $\sigma_1 \neq 0, \sigma_2<0$ and $\sigma_3 \leq 0$ are physical constants depending by the material. Taking the scaling transformations
$$
\tau=\frac{3 \sqrt{-\sigma_2}}{\sigma_1} t, \quad
\xi=\sqrt{-\sigma_2} x,
$$
and denoting $\gamma=3 \sigma_3 /\left(\sigma_1 \sigma_2\right)$, then relabeling as $u(t, x)=v(\tau, \xi)$, we reduce the above equation \eqref{or} to
\begin{align}\label{or1}
u_t-u_{x x t}+3 u u_x=\gamma\left(2 u_x u_{x x}+u u_{x x x}\right), \quad x \in \mathbb{R},\; t>0 .
\end{align}
In \cite{dai2}, the authors derived that value range of $\gamma$ is from -29.4760 to 3.4174 for some special compressible materials. From the mathematical
view point, we regard $\gamma$ as a real number. The real parameter $\gamma$ is related to the Finger deformation tensor of the material.
Both positive and negative values $\gamma$ are admissible. We refer to \cite{dai1,dai2} for more details on the physical background and the mathematical derivation of the model.
For $\gamma=0$, the rod equation \eqref{or1} reduces to the well-known BBM equation \cite{bbm}, modeling surface waves in a canal. In this case the solutions exist globally, meaning that the BBM equation can model permanent waves, but is unsuitable for describing breaking waves.
For $\gamma=1$, the rod equation \eqref{or1} becomes the Camassa-Holm equation, modeling long waves in shallow water. Such equation marked an important development in nonlinear dynamics and is of great current interest. The Camassa-Holm equation is thus much better understood than Eq.\eqref{or1}. In fact, the Camassa-Holm equation is completely integrable \cite{Camassa,Constantin-P} with a bi-Hamiltonian structure \cite{Constantin-E,Fokas} and infinitely many conservation laws \cite{Camassa,Fokas}. Also, it admits exact peaked soliton solutions (peakons) of the form $u(x,t)=ce^{-|x-ct|}\;(c>0)$, which are orbitally stable \cite{Constantin.Strauss}. Another remarkable feature of the Camassa-Holm equation is the wave breaking phenomena: the solution remains bounded while its slope becomes unbounded in finite time \cite{Constantin}.

We complement the equation \eqref{or1} with initial data
\begin{align}\label{or2}
u(0,x)=u_0(x),
\end{align}
 and vanishing boundary conditions
 \begin{align}\label{or3}
 u(x)\to0,\quad |x|\to\infty.
 \end{align}
Such boundary conditions will be taken into account through an appropriate choice of the functional setting, guaranteeing the well-posedness of the Cauchy problem.

We denote by $G(x)=\fr12 e^{-|x|}$
the fundamental solution of the operator $1-\partial_x^2$, i.e., $(1-\partial_x^2)^{-1} f=G * f$. Let $m=u-u_{x x}$ be the potential of $u$. We thus have $u=G*m$, and $m$ satisfies
\bbal
\pa_tm+\gamma um_x+2\gamma m\pa_xu+3(1-\gamma)u\pa_xu=0.
\end{align*}
It is also convenient to rewrite the Cauchy problem associated with the equations \eqref{or1}-\eqref{or3} in the following weak form:
\begin{align}\label{r}
\begin{cases}
\pa_tu+\gamma u\pa_xu=-\pa_xG*\f(\frac{3-\gamma}{2}u^2+\frac\gamma2(\pa_xu)^2\g), &\quad (t,x)\in \R^+\times\R,\\
u(0,x)=u_0(x), &\quad x\in \R.
\end{cases}
\end{align}

\subsection{Known Well-posedness results}\label{subsec2}

For general $\gamma\in\R$, the rod equation was studied sketchily by the Constantin and Strauss in \cite{const}. The well-posedness, global solutions and blow-up phenomena of the Cauchy problem for \eqref{r} has been investigated by Yin \cite{yin} (see also \cite{zhou1,zhou2}). For any $\gamma \in \mathbb{R}$, by applying Kato's theory, the Cauchy problem for the rod equation is locally well-posed in $H^s$ with $s>3 / 2$. More precisely, if $u_0 \in H^s(\mathbb{R})$ with $s>3/2$, then there exists a maximal time $0<T^* \leq \infty$ and a unique solution $u \in C\left(\left[0, T^*\right), H^s\right) \cap C^1(\left[0, T^*\right), H^{s-1})$. Moreover, the solution $u$ depends continuously on the initial data.
It is also known that $u$ admits the invariants
$$
E(u)=\int_{\mathbb{R}}\left(u^2+u_x^2\right)\dd x\quad\text{and}\quad
F(u)=\int_{\mathbb{R}}\left(u^3+\gamma u u_x^2\right) \dd x.
$$
In particular, the invariance of the $H^1$-norm of the solution implies that $u(t,x)$ remains uniformly bounded up to the time $T^*$. On the other hand, if $T^*<\infty$ then $\limsup_{t \rightarrow T}\|u(t)\|_{H^s}=\infty(s>3/2)$ and the precise blowup scenario of strong solutions is the following:
$$
T^*<\infty \Longleftrightarrow \liminf_{t \rightarrow T^*}\left\{\inf_{x \in \mathbb{R}} \gamma u_x(t, x)\right\}=-\infty.
$$
In order to obtain simpler sufficient conditions to ensure the occurrence of wave breaking for \eqref{r}, Brandolese \cite{B-cmp} established a new wave-breaking criterion, i.e, local-in-space blowup criterion which only
involves the values of $u'_0(x_0)$ and $u_0(x_0)$ in a single point $x_0$ of the real line, which have been improved in \cite{B-jfa}.

\subsection{Motivations}\label{subsec3}
In this paper, we focus on the ill-posedness of the strong solution for the rod equation in the low-regularity Sobolev spaces $H^s$ with $s<3/2$. For the critical case $s=3/2$, Guo-Liu-Molinet-Yin \cite{Guo-Yin} showed that the Camassa-Holm equation is ill-posed in $H^{3/2}$ by proving the norm inflation. When considering the ill-posedness for the rod equation in critial Sobolev spaces $H^{3/2}$, we just need one point large enough. In fact, following Guo-Liu-Molinet-Yin's method, we have
\bbal
\frac{\dd}{\dd t}\|u(t,x)\|^2_{H^2}&\leq C\|\na u(t,x)\|_{B^0_{\infty,\infty}}\|u(t,x)\|^2_{H^2}\log_2\f(1+\|u(t,x)\|^2_{H^2}\g)
\\&\leq C\|u(t,x)\|_{H^\frac32}\|u(t,x)\|^2_{H^2}\log_2\f(1+\|u(t,x)\|^2_{H^2}\g),
\end{align*}
which implies
\bbal
\|\pa_xu(t,x)\|_{L^\infty}\les \|u(t,x)\|_{H^2}\leq \|u_0\|_{H^2}\exp\f\{\exp\f\{C\int^t_0\|u(\tau,x)\|_{H^\frac32}\dd \tau\g\}\g\}.
\end{align*}
Hence, once it hods that $|\pa_xu(t,x_0)|\rightarrow \infty$ as $t$ tends to $T^*$ for some $x_0\in\R$, we must have $\|u(t)\|_{H^{3/2}}\rightarrow \infty$ as $t$ tends to $T^*$. However, to study the ill-posedness in the low-regularity Sobolev spaces $H^s$ with $s<\frac32$, we not only need $\pa_xu(t,x_0)$ to be large enough for some $x_0$, but also may need $\pa_xu(t,x)$ to be large uniformly in the interval $x\in(x_0-\delta,x_0+\delta)$ for some $\delta>0$. For example, for the Camassa-Holm equation, we have
\bbal
\|u\|^2_{L^2}+\|\pa_xu\|^2_{L^2}=\|u_0\|^2_{L^2}+\|\pa_xu_0\|^2_{L^2}.
\end{align*}
This show that, although $\pa_xu(t,x_0)$ is large enough for some $x_0\in \R$, we obtain that $\|\pa_xu(t,x)\|_{L^2}$ is not large. Hence, the ill-posed mechanism in low-regularity Sobolev spaces is completely different from the critical case. Moreover, it is interesting that, how we can find the critical index $s_0$ such that the strong solution $u(t,x)$ satisfy $\|u(t)\|_{H^s}$ will large if $s> s_0$, and $\|u(t)\|_{H^s}$ is small if $s< s_0$.

On the other hand, based on the construction of special symmetric 2-peakon solutions, called peakon-antipeakons, of the form
$$
u(x, t)=p(t) e^{-|x+q(t)|}-p(t) e^{-|x-q(t)|},
$$
many authors used the peakon-type traveling wave solutions to prove the ill-posedness of the Camassa-Holm type equation in the low-regularity Sobolev spaces $H^s(\R)$ with $s<3/2$.
Byers \cite{Byers} proved that the Camassa-Holm equation is ill-posed in $H^s$ for $1<s < 3/2$ in the sense of norm inflation, where norm inflation for the Camassa-Holm equation in $H^1(\R)$ is impossible by the most important conserved quantity. Himonas, Holliman and Grayshan \cite{cpde} established the ill-posedness of the Degasperis-Procesi equation in $H^s(\R)$ with $s<3/2$. Himonas, Grayshan and Holliman \cite{jns16} used the method of peakon-antipeakon solutions to establish the ill-posedness of the b-equations in
$H^s(\R)$ when $b > 1$, and Novruzov \cite{nov} also considered the ill-posedness for the b-family
of equations with non-zero dispersion coefficient in $H^s(\R)$ when $b<1$. Himonas, Holliman and Kenig \cite{siam} obtained the ill-posedness for the Novikov equation. We should mention that, all the ill-posedness results mentioned above heavily rely on the construction of peakon-antipeakon solutions. However, much less work has been done on the rod equation \eqref{r} in the low-regularity Sobolev spaces, and it is not clear that whether the method of peakon-antipeakon solutions is applied to the rod equation. On the other hand, the reason for revisiting the rod \eqref{r} is that, for $C^\infty$-smooth initial data, the method of peakon-antipeakon solutions is truly invalid even for the Camassa-Holm equation and Degasperis-Procesi equation since the peakon-antipeakon solutions is not in $C^1$.
{\it
An interesting open problem is that the smooth solution of the rod equation \eqref{r} exhibits the norm inflation phenomena in the low-regularity Sobolev setting.} In this paper, we present a new method to prove that the smooth solution of the rod equation \eqref{r} in the low regularity setting also exhibits the norm inflation phenomena.

\subsection{Main results}\label{subsec4}
Now let us state our main results of this paper.
\begin{theorem}(Blow-up)\label{th1}
Assume that $u_0\in H^3(\R)$.
 \begin{enumerate}
   \item Let $\gamma> 0$. If there exists a point $x_0\in \R$ such that $u'_0(x_0)<-C_\gamma \|u_0\|_{H^1}$ with $C_\gamma=\sqrt{1+\frac{3}{\gamma}}$, the solution of the rod equation \eqref{r} blows up in finite time.
   \item Let $\gamma< 0$. If there exists a point $x_0\in \R$ such that $u'_0(x_0)>C_\gamma \|u_0\|_{H^1}$ with $C_\gamma=\sqrt{1-\frac{3}{\gamma}}$, the solution of the rod equation \eqref{r} blows up in finite time.
 \end{enumerate}
\end{theorem}

\begin{theorem}(Norm inflation for $H^\infty$-smooth data)\label{th2}
Let $\gamma\neq  0$ and $s\in(1,\frac32)$.
For any $\ep> 0$, there exists a solution $u(t,x)\in C^1([0,T_0];H^\infty(\R))$ of the Cauchy problem \eqref{r} and $0 < T_0 < \ep$ such that $u_0 \in {\cal S}(\R)$ satisfying
\begin{align*}
\|u_0\|_{H^s}\leq \ep\quad\text{but}\quad\|u(T_0)\|_{H^s} > \frac{1}{\ep}.
\end{align*}
\end{theorem}
\begin{remark}\label{re1}
Theorem \ref{th2} indicates the ill-posedness of \eqref{r} in $H^s(\R)$ with $s\in(1,\frac32)$ in the sense that the solution map $H^s\ni u_0 \mapsto u\in H^s$ is strong instability with respect to the initial data. We should mention that, for the sharp end-point case of $H^1$, Theorem \ref{th2} is impossible by the conservation of energy. Moreover, this theorem is a very interesting result in its own right. Indeed, those smooth solutions exhibit the same norm inflation phenomena with the weak ones considered in literature \cite{Byers,cpde,jns16,nov,siam}.
\end{remark}
\begin{remark}\label{re2}
Compared with the method of peakon-antipeakon solutions for proving the norm inflation phenomena, our idea is completely different. We mention that the method used here for proving ill-posedness for the rod equation is suitable for other shallow water wave equations.
\end{remark}
\begin{remark}
We summary the well-posedness/ill-posdeness results of the rod equation in the Sobolev spaces. This can be seen clearly from the Table 1.
\end{remark}
\begin{table}[http]
      \centering
      \begin{tabular}{l|c|c}\hline
       References&Range&Results\\\hline
        \cite{mz2000,const,yin,zhou1,zhou2}&\makecell[c]{$s>\fr32$}&LWP \\\hline
        \cite{Guo-Yin,guo}&$s=\fr32$&Norm inflation
        \\\hline
        Theorem \ref{th1}&$1<s<\fr32$& Norm inflation
          \\\hline
        Lemma \ref{le2}&$s=1$& No Norm inflation \\\hline
        \end{tabular}
        \caption{Well/Ill-posedness of \eqref{r} in $H^s$}
        \end{table}
\subsection{Sketch of the Key Ideas}\label{subsec5}
Let us make comments on the main ideas in proving Theorem \ref{th2}. When proving the norm inflation in the low-regularity setting, we meet two major difficulties in our analysis: the first one appears from the constructions of initial data and the other is to develop the structure of the equation which leads to the norm inflation.
Firstly, our new contribution is to construct the initial data which satisfies that
\begin{itemize}
  \item $\|u_0\|_{H^s}$ is equivalent to $\|u_0\|_{W^{1,p_s}}$ with $p_s=\frac{1}{\fr32-s}$;
  \item the principal part of  $\|u_0\|_{W^{1,p_s}}$ is  $\f(\int^{q_0}_{-q_0}|\pa_xu_0|^{p_s}\dd x\g)^{\frac1{p_s}}$.
\end{itemize}
Secondly, based on the construction of initial data, we find that the lifespan $T^*$ of solution satisfies that $T_{\min}\leq T^*\leq T_{\max}$, where $T_{\max}$ is close to $T_{\min}$. Furthermore, we make an important observation that, $\pa_xu(t,x)$ is large enough uniformly in $(t,x)\in(0,T^*)\times(\phi(t,-q_0),\phi(t,q_0))$, where $\phi(t,x)$ is the flow map induced by the solution $u$. Finally, we find that the norm inflation of solution comes from the crucial quantity $\f(\int^{\phi(t,q_0)}_{\phi(t,-q_0)}|\pa_xu|^{p_s}\dd x\g)^{\frac{1}{p_s}}$.
In particular, these two key observations rely on the continuity of solution to the rod equation. To overcome the difficulty, we can use the smooth solution to the rod equation with the mollified smooth version of $u_0$ and proceed with the above steps with minor modifications.
\subsection{Organization of our paper}\label{subsec6}
In Section \ref{sec2}, we list some notations and recall some known results which will be used in the sequel. In Section \ref{sec3} we present the proof of Theorem \ref{th1} by proving the blow-up phenomena. In Section \ref{sec4} we prove norm inflation  phenomena by dividing it into several parts:
(1) The Example Class for the Non-smooth Initial Data;
(2) Construction of Smooth Initial Data;
(3) Upper and Lower Bound of Lifespan;
(4) Norm Inflation for Smooth Initial Data.

\section{Preliminaries}\label{sec2}
\subsection{Notation}
For the reader's convenience, we  collect some notation and recall a few function spaces and norms that will be used frequently throughout this paper.
\begin{itemize}
  \item We use the letter $C$ to  denote generic constants whose
value can change from one line to another or even within a single line. $C_{s,t}$ stands for a positive constant that may still depend on $s$ and $t$. The symbol $A\approx B$ means that $C^{-1}B\leq A\leq CB$.
  \item Given a Banach space $X$, we denote its norm by $\|\cdot\|_{X}$.
  \item We denote by $W^{1,p}$ the standard Sobolev space on $\R$ of $L^p$ functions whose derivative also belongs to $L^p$.
\item For $I\subset\R$, we denote by $C(I;X)$ the set of continuous functions on $I$ with values in $X$. Sometimes we will denote $L^p(0,T;X)$ by $L_T^p(X)$.
  \item  We use $\mathcal{S}(\R)$ to denote Schwartz functions spaces on $\R$. Let us recall that for all $f\in \mathcal{S}$, the Fourier transform $\widehat{f}$ is defined by
$$
\widehat{f}(\xi)=\int_{\R}e^{-ix\xi}f(x)\dd x \quad\text{for any}\; \xi\in\R.
$$
  \item {\bf Sobolev space}. For $s \in \mathbb{R}$, the inhomogeneous Sobolev space $H^s=H^s(\mathbb{R})$ is defined as

$$
H^s(\mathbb{R}) :=\left\{f \in L^2(\mathbb{R}):\|f\|_{H^s(\mathbb{R})}<+\infty\right\},
$$
with norm
$$
\|f\|_{H^s}=\|f\|_{L^2}+\|f\|_{\dot{H}^s},
$$
where $\|f\|_{\dot{H}^s}=\left\|\Lambda^s f\right\|_{L^2}$ and $\Lambda^s$ is defined by $\widehat{\Lambda^s f}(\xi)=|\xi|^s \widehat{f}(\xi)$.
\item {\bf Friedrichs mollifiers}. Let $J(x)$ be a non-negative $C^\infty(\mathbb{R})$ function supported in the interval $[-1, 1]$ with the form
\bbal
J(x)=
\begin{cases}
\frac{1}{c}e^{\frac{1}{x^2-1}}, & |x|< 1,\\
0, & |x|\geq 1,
\end{cases}
\end{align*}
where we choose
$c=\int_{|x|<1}e^{\frac{1}{x^2-1}}\dd x$ such that $\int_{\R}J(x)\dd x=1$.

The operator $J_\ep$ is  the Friedrichs mollifier
defined by $J_{\ep}(x)=\frac1\ep J(\frac x\ep)$ for each $\ep\in (0, 1]$ and $x \in \mathbb{R}$, then it is not difficult to verify that $J_{\varepsilon}$ is supported in the ball of radius $\varepsilon$ about the origin and $J_{\varepsilon} u$ is real-valued for any real-valued function $u$.
\end{itemize}

\subsection{Known Lemmas}
The following Lemmas have been established in different contents. We refer to see the references \cite{mz2000,const,yin,zhou1,zhou2}.
\begin{lemma}\label{le1} Let $u_0 \in H^s(\mathbb{R})$ and $s>\frac{3}{2}$. Then there exists a time $T=T\left(\left\|u_0\right\|_{H^s}\right)>0$ to ensure that the rod equation \eqref{r} possesses a unique solution
$$
u\in C\left([0, T) ; H^s(\mathbb{R})\right) \cap C^1\left([0, T) ; H^{s-1}(\mathbb{R})\right).
$$
Moreover, the solution $u$ depends continuously on the initial data $u_0$.
\end{lemma}
The maximum value of $T^*$ in Lemma \ref{le1} is called the maximum existence time or the lifespan of the solution in general. If $T^*<\infty$ , that is $\limsup_{t \rightarrow T^*} \|u\|_{H^s}=\infty$, we say that the solution blows up in finite time.
\begin{lemma}\label{le2} Let $u_0 \in H^s(\mathbb{R})$ and $s>\frac{3}{2}$. If $u$ satisfies the rod equation \eqref{r}, then
$$
\int_{\mathbb{R}}\left(u^2+(\pa_xu)^2\right) \dd x=\int_{\mathbb{R}}\left(u_0^2+(\pa_xu_0)^2\right) \dd x.
$$
In particular, $$
\|u\|_{L^{\infty}} \leqslant \frac{1}{\sqrt{2}}\left\|u_0\right\|_{H^1}.
$$
This tells us that, the invariance of the Sobolev $H^1$-norm of the solution implies that $u(x, t)$ remains uniformly bounded up to the time $T^*$.
\end{lemma}

\begin{lemma}\label{le3} Let $u_0 \in H^s(\mathbb{R})$ and $s>\frac{3}{2}$. Suppose that $T^*$ is the lifespan of the solution for the rod equation \eqref{r}. If $T^*<\infty$, then $\limsup _{t \rightarrow T^*}\|u(t)\|_{H^s}=\infty$ and the precise blowup scenario of strong solutions is the following:
$$
T^*<\infty \Longleftrightarrow \liminf _{t \rightarrow T^*}\left\{\inf _{x \in \mathbb{R}} \gamma \pa_xu(t, x)\right\}=-\infty.
$$
\end{lemma}
We consider the following ODE to find the flow map $\phi$ induced by the Lipschitz field $u$:
\begin{align}\label{ode}
\quad\begin{cases}
\frac{\dd}{\dd t}\phi(t,x)=\gamma u(t,\phi(t,x)),\;&(t,x)\in(0,T^*)\times\R,\\
\phi(0,x)=x,&x\in\R,
\end{cases}
\end{align}
where $T^*$ is the maximal existence time of the solution defined as in Lemma \ref{le3}.

Because the velocity field is Lipschitz, then we get that for $t\in[0,T^*)\times\R$
\bbal
\pa_x\phi(t,x)=\exp\left(\gamma\int^t_0\pa_x u(\tau,\phi(\tau,x))\dd \tau\right)>0.
\end{align*}
This shows that $\phi(t,\cdot)$ is an increasing diffeomorphism over $\R$, that is,  for all $x,y\in \R,$ there holds that $\phi(t,x)<\phi(t,y)$ if $x< y$. Applying the classical result in the theory of ordinary differential equations, we have
\begin{lemma}\label{le4} Assume $u_0 \in H^s(\mathbb{R})$ and $s>\frac{3}{2}$. If $T^*$ is defined as in Lemma \ref{le3}, then the ODE \eqref{ode} possesses a unique solution $\phi \in C^1([0, T) \times \mathbb{R}, \mathbb{R})$. In addition, the map $\phi(t, \cdot):\R\rightarrow\R$ is an increasing diffeomorphism of the line for every $t\in[0,T^*)$ satisfying $\pa_x\phi(t, x)>0$ in the domain $[0, T^*) \times \mathbb{R}$.
\end{lemma}

The following lemma concerns the Riccati-type differential inequality.

\begin{lemma}\label{le5} Let $a, b,c>0$ and $\lambda=\sqrt{\frac{b}{a}}c$. If $f(t) \in C^1(\mathbb{R})$ satisfies
\begin{align}\label{ineq-ode}
\begin{cases}
f^{\prime}(t) \leq-a f^2(t)+bc^2,&\quad t>0,\\
f(0)<-\lambda,
\end{cases}
\end{align}
then
$$
f(t)\leq\left(at+\frac{1}{f(0)+\lambda }\right)^{-1}-\lambda.
$$
Furthermore, it holds
$$
f(t) \rightarrow-\infty, \quad \text { as } \quad t \rightarrow T^* \leq -\frac{1}{a} \frac{1}{f(0)+\lambda}.
$$
\end{lemma}
\begin{proof} We first claim that $f^{\prime}(t)<0$ for all $t>0$.
If not, by continuity of $f(t)$, there exists $t_0 \in[0, T)$ such that $f^{\prime}(t)<0$ for all $t \in\left[0, t_0\right)$ and $f^{\prime}\left(t_0\right)=0$. Since $f$ is decreasing in this interval, we have $f\left(t_0\right) \leq f(0)<-\lambda$. From $\eqref{ineq-ode}_1$, we find that $f^{\prime}\left(t_0\right) \leq -a f^2(0)-bc^2<0$, which leads to a contradiction.

Note that $f^{\prime}(t)<0$ for $t>0$ implies that $f(t)\leq f(0)<-\lambda$ for $t>0$, thus
$$
f^{\prime} \leq -a f^2+bc^2 \leq-a\left(f+\lambda\right)^2.
$$
Solving this inequality, we derive that
$$
f(t)\leq\left(at+\frac{1}{f(0)+\lambda }\right)^{-1}-\lambda.
$$
Thus we complete the proof.
\end{proof}

In this paper, we focus on the Cauchy problem of rod equation \eqref{r} only for the case $\gamma>0$. In fact, for the case $\gamma<0$, taking the transformations $v=-u$ and $\tilde{\gamma}=-\gamma$, we can consider the new Cauchy problem
\begin{align*}
\begin{cases}
v_t+\tilde{\gamma} v \pa_xv=-\partial_x G*\left(-\frac{3+\tilde{\gamma}}{2} v^2+\frac{\tilde{\gamma}}{2} (\pa_xv)^2\right), &\quad (t,x)\in \R^+\times\R,\\
v(0,x)=v_0(x), &\quad x\in \R.
\end{cases}
\end{align*}
\section{Theorem \ref{th1}: Blow-up}\label{sec3}

For $x\in\R$, let
\begin{align}\label{ODE}
\frac{\dd\phi}{\dd t}(t,x)=\gamma u(t,\phi(t,x)),\quad \phi(0,x)=x,
\end{align}
Since $u(x, t)$ is bounded and satisfies a Lipschitz condition in $x$ for any $x\in\R$ for any
$(0, T^*)$, it follows from ODE theory that \eqref{ODE} possesses a unique solution in $C^1(0, T^*)$
for any $x\in\R$. Furthermore, $x \mapsto\phi(\cdot ; x)$ is infinitely continuously differentiable throughout the interval $(0, T^*)$ for
any $x\in\R$.
Differentiating \eqref{r} with respect to space variable $x$, we find
\bal\label{u1}
\pa_t\pa_xu+\gamma u\pa^2_xu+\frac\gamma2(\pa_xu)^2=\frac{3-\gamma}{2}u^2-(1-\pa^2_x)^{-1}\f(\frac{3-\gamma}{2}u^2+\frac\gamma2(\pa_xu)^2\g)=:V.
\end{align}
Combining \eqref{ode} and \eqref{u1}, we obtain
\bal\label{du1}
\frac{\dd}{\dd t}\pa_xu(t,\phi(t,x))
&=-\frac\gamma2(\pa_xu)^2(t,\phi(t,x))+V(t,\phi(t,x)).
\end{align}
Notice that, the convolution kernel for $(1-\partial^2_{x})^{-1}$ is denoted by $G(x)=\fr12e^{-|x|}$,  we have
\bal\label{du2}
\f\|V\g\|_{L^{\infty}}=&~\f\|\frac{3-\gamma}{2}u^2-(1-\pa^2_x)^{-1}\f(\frac{3-\gamma}{2}u^2+\frac\gamma2(\pa_xu)^2\g)\g\|_{L^{\infty}}\nonumber\\
=&~\frac{|3-\gamma|}{2}\f\|u\g\|^2_{L^{\infty}}+\left\|G*\f(\frac{3-\gamma}{2}u^2+\frac\gamma2(\pa_xu)^2\g)\right\|_{L^{\infty}}\nonumber\\
\leq&~ \frac{3+\gamma}{2}\|u_0\|^2_{H^1},
\end{align}
where we have used
\bbal
\|u(t)\|_{L^\infty}\leq \frac{1}{\sqrt{2}}\|u(t)\|_{H^1}=\frac{1}{\sqrt{2}}\|u_0\|_{H^1}
\end{align*}
and the invariance of the Sobolev $H^1$-norm of the solution in the last step.

From \eqref{du1} and \eqref{du2}, we deduce that
\begin{align*}
\frac{\dd}{\dd t}\f(\gamma\pa_xu(t,\phi(t,x))\g)\leq -\frac12(\gamma\pa_xu)^2(t,\phi(t,x))+\frac{\gamma(3+\gamma)}{2}\|u_0\|_{H^1}^2.
\end{align*}
Notice that $u'_0(x_0) < -\sqrt{\frac{3+\gamma}{\gamma}}\|u_0\|_{H^1}$, by Lemma \ref{le5}, we deduce that
$$\inf_{x \in \mathbb{R}} \f\{\gamma \pa_xu(t, x)\g\}=\inf _{x \in \mathbb{R}} \f\{\gamma \pa_xu(t,\phi(t,x))\g\}\to-\infty
\quad\text{as}\;t \rightarrow T^*\leq\frac{-2}{\gamma u'_0(x_0) +\sqrt{\gamma(3+\gamma)}\|u_0\|_{H^1}},$$
which implies that the solution to \eqref{r} must blow up in finite time.
Thus, Theorem \ref{th1} is proved.

\section{Theorem \ref{th2}: Norm Inflation}\label{sec4}
\subsection{The Example Class for the Non-smooth Initial Data}\label{subsec41}
Firstly, we construct an explicit example as follows. Let $p_0\gg1$ and $0<q_0\ll1$ which will be fixed later. Set
\bal\label{ini0}
v_0(x)=
\begin{cases}
p_0q_0e^{x}, &\;\text{if}\; x\in(-\infty,-q_0),\\
-p_0e^{-q_0}x, &\;\text{if}\; x\in[-q_0,q_0],\\
-p_0q_0e^{-x}, &\;\text{if}\; x\in(q_0,+\infty).
\end{cases}
\end{align}
That is, the initial profile $v_0(x)$ is the anti-symmetric peakon-antipeakon (see {\bf Fig.1}).
\vskip0.1in
\hskip1in
\begin{minipage}{0.7\linewidth}
\hspace*{0cm}
\vspace*{0cm}
\begin{tikzpicture}[xscale=1,yscale=1]
%
%
\newcommand\X{7};
\newcommand\Y{2};
\newcommand\FX{11};
\newcommand\FY{11};
\newcommand\Z{0.6};
\newcommand\C{5};
\newcommand\A{1};
%
%
\draw[->,line width=1pt,black] (-5,0)--(5,0)
node[above left] {\fontsize{\FX}{\FY}$x$};
\draw[->,line width=1pt,black] (0,-2.5)--(0,2.5) node[below right] {\fontsize{\FX}{\FY}$v_0$};
\draw[domain=-4:4, variable=\x,
red, line width=1.5pt]
plot ({\x},{(\x < -\A) * (\C * \A * exp(\x)) +
                (\x >= -\A && \x < \A) * (-\C * exp(-\A) * \x) +
                (\x >= \A) * (-\C * \A * exp(-\x))});
\draw[line width=1pt,black,dashed]
({-(\A)},{\C * \A * exp(-\A)})
node[] { }
node[above,xshift=-1.0cm] {\fontsize{\FX}{\FY}
$-v_0(q_0)=v_0(-q_0)$}
--
({-(\A)},0)
node[below, xshift=-.2cm] {\fontsize{\FX}{\FY}$-q_0$}
node[] {$\bullet$} ;
\draw[line width=1pt,black,dashed]
({(\A)},{-\C * \A * exp(-\A)})
node[] {}
node[below,xshift=1.2cm] {\fontsize{\FX}{\FY}
$v_0(q_0)=-p_0q_0e^{-q_0}$}
--
({(\A)},0)
node[above, xshift=.2cm] {\fontsize{\FX}{\FY}$q_0$} node[] {$\bullet$};
\end{tikzpicture}
\end{minipage}
\vskip0.15in
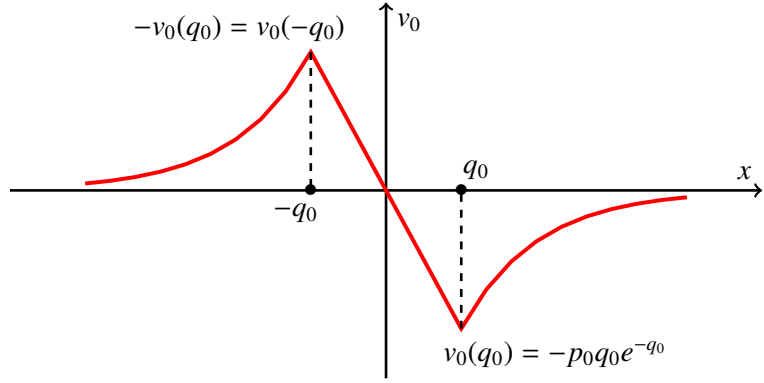
\centerline{{\bf Fig.1}: Initial profile $v_0(x)$.}
\vskip0.1in
From {\bf Fig.1}, it is easy to see that $v_0(x)$ is an odd function and
$$\|v_0\|_{L^\infty(\R)}\leq p_0q_0,$$
and
\bal\label{ini-2}
v'_0(x)=
\begin{cases}
p_0q_0e^{x}, &\;\text{if}\; x\in(-\infty,-q_0),\\
-p_0e^{-q_0}, &\;\text{if}\; x\in(-q_0,q_0),\\
p_0q_0e^{-x}, &\;\text{if}\; x\in(q_0,+\infty),
\end{cases}
\end{align}
which is displayed in {\bf Fig.2}.
\vskip0.1in
\hskip1in
\begin{minipage}{0.7\linewidth}
\hspace*{0cm}
\vspace*{0cm}
\begin{tikzpicture}[xscale=1,yscale=1]
%
%
\newcommand\X{7};
\newcommand\Y{2};
\newcommand\FX{11};
\newcommand\FY{11};
\newcommand\Z{0.6};
\newcommand\C{2};
\newcommand\A{1};
%
%
\draw[->,line width=1pt,black] (-5,0)--(5,0)
node[above left] {\fontsize{\FX}{\FY}$x$};
\draw[->,line width=1pt,black] (0,-2.5)--(0,2.5) node[below left] {\fontsize{\FX}{\FY}$v'_0$};
\draw[domain=-4:-1, variable=\x,
red, line width=1.5pt]
plot ({\x},{\C*exp(\x+\A)-\C*exp(\x-\A)});
\draw[domain=-1:1, variable=\x,
red, line width=1.5pt]
plot ({\x},{-3*\C*exp(-\A)});
\draw[domain=1:4, variable=\x,
red, line width=1.5pt]
plot ({\x},{\C*exp(-\x+\A)-\C*exp(-\x-\A)});
\draw[line width=1pt,black,dashed]
({(\A)},{-3*\C*exp(-\A)})
node[] {}
node[below right,xshift=-.2cm] {\fontsize{\FX}{\FY}
$v'_0(q_0^-)=-p_0e^{-q_0}$}
--({(\A)},0)
node[below, xshift=.4cm] {\fontsize{\FX}{\FY}$q_0$} node[] {$\bullet$};
\draw[line width=1pt,black,dashed]
({-(\A)},{\C-\C*exp(-2*\A)})
node[] { }
node[above,xshift=-1.0cm] {\fontsize{\FX}{\FY}}
--
({-(\A)},0)
node[below, xshift=-.4cm] {\fontsize{\FX}{\FY}$-q_0$}
node[] {$\bullet$} ;
\draw[line width=1pt,black,dashed]
({-(\A)},{-3*\C*exp(-\A)})
node[] { }
node[above,xshift=-1.0cm] {\fontsize{\FX}{\FY}}
--
({-(\A)},0)
node[below, xshift=-.4cm] {\fontsize{\FX}{\FY}$-q_0$}
node[] {$\bullet$} ;
\draw[line width=1pt,black,dashed]
({(\A)},{\C-\C*exp(-2*\A)})
node[] {}
node[above right,xshift=-.2cm] {\fontsize{\FX}{\FY}
$v'_0(q_0^+)=p_0q_0e^{-q_0}$}
--
({(\A)},0)
node[below, xshift=.4cm] {\fontsize{\FX}{\FY}$q_0$} node[] {$\bullet$};
\end{tikzpicture}
\end{minipage}
\vskip0.1in
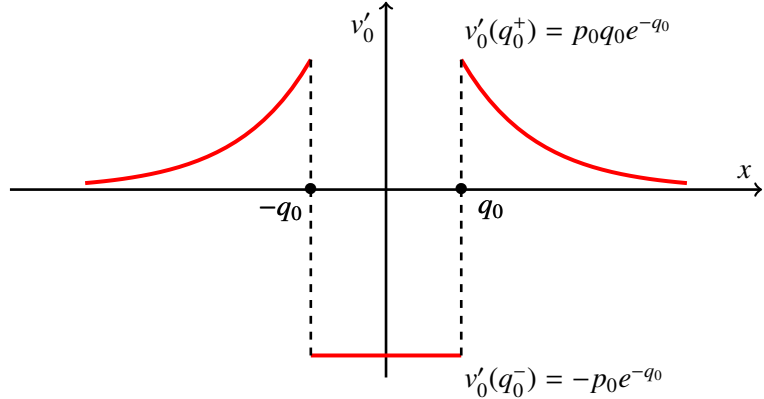
\centerline{{\bf Fig.2}: Graph of $v'_0(x)$.}

Furthermore, we can deduce that the following result holds:
\begin{lemma}\label{le001} Let $v_0$ be given by \eqref{ini0}. For any $p_0>0, q_0\in\f(0,\fr14\g)$ and $s\in\left(\frac{1}{2}, \frac{3}{2}\right)$, there exists a constant $C=C_s>0$ independent of $p_0$ and $q_0$ such that
\begin{align*}
C^{-1} p_0q_0^{3/2-s}\leq\|v_0\|_{H^s}\leq C p_0q_0^{3/2-s}.
\end{align*}
\end{lemma}
\begin{proof} It is not difficult to verify that
\bal\label{hs00}
\|v_0\|_{L^2}=p_0q_0e^{-q_0}\f(1+\frac23q_0\g)^{\fr12}.
\end{align}
Now we claim that, there exists a constant $C=C_s > 0$, independent of $p_0$ and $q_0$, such that
\bal\label{hs0}
C^{-1} p_0q_0^{\frac{3}{2}-s} \leq \|v_0\|_{\dot{H}^s} \leq C p_0q_0^{\frac{3}{2}-s}.
\end{align}
From \eqref{hs00} and \eqref{hs0}, we deduce that for some $C=C_s>0$, independent of $p_0$ and $q_0$, such that
\begin{align*}
C^{-1} p_0q_0^{\frac32-s}\leq\|v_0\|_{H^s}\leq C p_0q_0^{\frac32-s}.
\end{align*}
Next we aim to show \eqref{hs0}. Taking the Fourier transform of $v_0$ with respect to $x$ yields
\bbal
\hat{v}_0(\xi)&
=\f(\int_{-\infty}^{-q_0} +\int_{-q_0}^{q_0}+\int_{q_0}^{+\infty}\g)v_0(x) e^{-i\xi x} \dd x\nonumber\\
&=:I_1+I_2+I_3,\end{align*}
where
\bbal
I_1 & = p_0 q_0 \int_{-\infty}^{-q_0} e^{(1-i\xi)x} \dd x
= p_0 q_0 e^{-q_0}\frac{ e^{i\xi q_0}}{1-i\xi},
\\
I_2 &=-p_0 e^{-q_0} \int_{-q_0}^{q_0} x e^{-i\xi x}\dd x
= -2ip_0 e^{-q_0} \frac{q_0\xi\cos(q_0\xi)-\sin(\xi q_0)}{\xi^2} ,
\\I_3 &= -p_0 q_0\int_{q_0}^{+\infty} e^{-x} e^{-i\xi x} \dd x
= -p_0 q_0 e^{-q_0} \frac{e^{-i\xi q_0}}{1+i\xi},
\end{align*}
then we have
\bbal
\hat{v}_0(\xi) = 2i p_0 e^{-q_0} \left(
\frac{q_0 \sin(q_0\xi)}{1 + \xi^2}  + \frac{\sin(q_0\xi)}{\xi^2}-\frac{q_0 \cos(q_0\xi)}{(1 + \xi^2)\xi}
\right).
\end{align*}
We denote
$$
\widehat{f}(\xi) =
\frac{q_0 \sin(q_0\xi)}{1 + \xi^2}  + \frac{\sin(q_0\xi)}{\xi^2}-\frac{q_0 \cos(q_0\xi)}{(1 + \xi^2)\xi}.
$$
Our aim is to reduce that, there exists a constant $C=C_s > 0$, independent of $p_0$ and $q_0$, such that
\bal\label{hs1}
C^{-1}q_0^{3-2s} \leq \|f\|^2_{\dot{H}^s} \leq Cq_0^{3-2s}.
\end{align}
Using the definition of the $\dot{H}^s$-norm and the change of variable setup $\eta=q_0\xi$, we have
\bal\label{hs2}
\|f\|_{\dot{H}^s}^2 &=
 2\int^\infty_0|\xi|^{2s}\left|\frac{q_0 \sin(q_0\xi)}{1 + \xi^2}  + \frac{\sin(q_0\xi)}{\xi^2}-\frac{q_0 \cos(q_0\xi)}{(1 + \xi^2)\xi} \right|^2 \dd\xi\nonumber
 \\&=2q_0^{3-2s} \int_{0}^{\infty} |\eta|^{2s-2} \left| \frac{q_0 (\eta\sin \eta-q_0\cos\eta)}{q_0^2 + \eta^2} + \frac{\sin \eta}{\eta}\right|^2 \dd\eta.
\end{align}
Define
$$F:=\int_{0}^{\infty}|\eta|^{2s-2} \left| \frac{q_0 (\eta\sin \eta-q_0\cos\eta)}{q_0^2 + \eta^2} + \frac{\sin \eta}{\eta}\right|^2\dd\eta,$$
we just need to prove that, there exists a constant $C=C_s > 0$, independent of $p_0$ and $q_0$, such that
\bal\label{hsyy}
C^{-1} \leq F \leq C.
\end{align}

{\bf Lower  bound.} Obviously, one has
\bal\label{hs3}
 \left| \frac{q_0 (\eta\sin \eta-q_0\cos\eta)}{q_0^2 + \eta^2} + \frac{\sin \eta}{\eta}\right|\geq \f|\frac{\sin \eta}{\eta}\g|-\frac{q_0 |\eta\sin \eta|}{q_0^2 + \eta^2}-\frac{q^2_0|\cos\eta|}{q_0^2 + \eta^2}.
\end{align}
Using the basic fact $\frac2\pi x\leq \sin x\leq x$ for all $x\in[0,\frac{\pi}{2}]$, then we have for all $\eta\in[\frac\pi4,\frac\pi2]$
\bbal
\text{RHS of \eqref{hs3}}&= \frac{\sin \eta}{\eta}-\frac{q_0 \eta\sin \eta}{q_0^2 + \eta^2}-\frac{q^2_0\cos\eta}{q_0^2 + \eta^2}\\
&\ge \f(1-q_0\g) \frac{\sin \eta}{\eta}-\frac{q^2_0 }{\eta^2}
\\&\geq \frac2\pi\f(1-q_0-\frac{8}{\pi}q^2_0 \g)\geq \frac1\pi.
\end{align*}
Thus one has
\begin{align*}
F&\geq \frac1\pi\int_{\frac\pi4}^{\frac\pi2}\eta^{2s-2} \dd \eta=\frac{\pi^{2s-2}}{2s-1}\f(2^{1-2s}-4^{1-2s}\g).
\end{align*}
From which and \eqref{hs2}, we obtain the lower  bound.

{\bf Upper  bound.} We split the domain of integration as
\begin{align*}
F&=\left(\int_{0}^{1}+\int_{1}^{\infty}\right)\eta^{2s-2} \left| \frac{q_0 (\eta\sin \eta-q_0\cos\eta)}{q_0^2 + \eta^2} + \frac{\sin \eta}{\eta}\right|^2\dd\eta
=:F_1+F_2.
\end{align*}
Due to the simple fact $\sin x\leq \min\{x,1\}$ for $x\geq0$, we have
\begin{align*}
&F_1\leq 4\int_{0}^{1}\eta^{2s-2} \dd \eta\leq \frac{4}{2s-1},\nonumber\\
&F_2\leq 9\int_{1}^{\infty}\eta^{2s-4}\dd \eta \leq \frac{9}{3-2s},
\end{align*}
where we have used $s\in\left(\frac{1}{2}, \frac{3}{2}\right)$ and $q_0 \in(0,1)$. This proves the upper bound.

Then we complete the proof of Lemma \ref{le001}.
\end{proof}

\subsection{Construction of Smooth Initial Data}\label{subsec42}

Note that $v_0\notin C^1(\R)$, to obtain the smooth solution to the rod equation \eqref{r} with smooth initial data,  we will use the mollified smooth version of $v_0$ as the initial data $u_0$. More precisely, we set
\bal\label{de-u0}
u_0(x)=(J_{q^2_0}\ast v_0)(x)=\int_{\R}v_0(x-y)J_{q^2_0}(y)\dd y.
\end{align}
Then, $u_0(x)$ is a real-valued odd function and $u_0(x)\in C^\infty(\mathbb{R})$ since $v_0(x)\in C(\mathbb{R})$. We will use the simple and key observation that
$$u'_0(x)=v'_0(x)=-p_0e^{-q_0},\quad\text{for}\; x\in(-q_0+q_0^2,q_0-q_0^2),$$
which gives that for $p\in(1,\infty)$
\bal\label{de-u10}
\|u'_0(x)\|^p_{L^p(|x|\leq q_0-q_0^2)}=2(q_0-q_0^2)p_0^pe^{-pq_0}.
\end{align}
\begin{remark}\label{remark}
For the rod equation \eqref{r}, the above smooth initial data $u_0$ must develop singularity in finite time. In fact, by Lemma \ref{le5}, we know that
the corresponding strong solution to \eqref{r} blows up in finite time. Moreover, the maximal time of
existence $T^*$ satisfies that $T^*\leq -\frac{2}{\gamma u'_0(0)}=\frac{2e^{q_0}}{\gamma p_0}$.
\end{remark}
Furthermore, we can deduce that the following result holds:
\begin{corollary}\label{co1} For any $p_0>0, q_0\in\f(0,\fr14\g)$ and $s\in\f(\fr12,\fr32\g)$, there exists $C=C_s>0$ independent of $p_0$ and $q_0$ such that
\begin{align*}
C^{-1} p_0q_0^{3/2-s} \leq \|u_0\|_{H^s}\leq C p_0q_0^{3/2-s}.
\end{align*}
\end{corollary}
\begin{proof}
Due to $H^s\hookrightarrow W^{1,p_s}$ with $p_s=\frac{1}{\frac32-s}$, and using \eqref{de-u10}, we have
\bbal
C\|u_0\|_{H^s}\geq \|u'_0\|_{L^{p_s}(|x|\leq q_0-q_0^2)}\geq(q_0-q_0^2)^{\frac32-s}p_0e^{-q_0}\geq cp_0q_0^{3/2-s}.
\end{align*}
From \eqref{de-u0} and Lemma \ref{le001}, we also have
\bbal
\|u_0\|_{H^s}&\leq C\|v_0\|_{H^s}\leq C p_0q_0^{3/2-s}.
\end{align*}
Thus we obtain the desired result.
\end{proof}
\subsection{Upper and Lower Bound of Lifespan}
From now on, we assume that the Cauchy problem \eqref{r} with initial data $u(0,x) = u_0(x)$ given by \eqref{de-u0}
possesses a unique solution $u\in C([0, T^*); H^\infty(\R))$ for some $T^*> 0$, where $H^\infty(\R)=\cap_{s=1}^\infty H^s(\R)$. As a matter of
fact, one may combine an a priori bound and a compactness argument to work out
the local-in-time well-posedness in $H^s(\R)$ for $s > 3/2$. From \eqref{r} and standard product laws in Sobolev spaces, we deduce that $u\in C^1([0, T^*); H^\infty(\R))$.
This will be more than enough to make the arguments in the following steps rigorous. We assume that $T^*$ is the maximal time of existence or the lifespan of solution.
Then we have
\begin{lemma}\label{life}
The lifespan $T^*$ of solution to the rod equation \eqref{r} with initial data $u_0(x)$ given by \eqref{de-u0} satisfies that
\bbal
\frac{1}{\gamma p_0}\leq T^*\leq \frac3{\gamma p_0}.
\end{align*}
\end{lemma}
\begin{proof}\,Indeed, from Remark \ref{remark}, we obtain the upper bound of $T^*$.

Let $\phi(t,x)$ be the flow map which solves \eqref{ODE}.
 From \eqref{du1} and \eqref{du2}, we have for $(t,x)\in[0,T^*)\times \R$
\begin{align}\label{ODE1}
\begin{cases}
\frac{\dd}{\dd t}\pa_xu(t,\phi(t,x))\leq -\frac\gamma2(\pa_xu)^2(t,\phi(t,x))+\frac{3+\gamma}{2}\|u_0\|_{H^1}^2,\\
\frac{\dd}{\dd t}\pa_xu(t,\phi(t,x))\geq -\frac\gamma2(\pa_xu)^2(t,\phi(t,x))-\frac{3+\gamma}{2}\|u_0\|_{H^1}^2.
\end{cases}
\end{align}

{\bf Case 1: $x\in \mathbf{I}_1=[-q_0+q^2_0,q_0-q^2_0]$}.

When $x\in[-q_0+q^2_0,q_0-q^2_0]$ and $p_0$ is large enough, it holds that 
$$u'_0(x)= u'_0(0)< -{C}_{\gamma}\|u_0\|_{H^1}\quad \text{with}\quad {C}_{\gamma}=\sqrt{\frac{3+\gamma}{\gamma}}.$$
By argument of continuity, we have
\bbal
\pa_xu(t,\phi(t,x))\leq u'_0(x)< -{C}_{\gamma}\|u_0\|_{H^1}, \quad \forall (t,x)\in[0,T^*)\times \mathbf{I}_1,
\end{align*}
which implies
\begin{align}\label{ODE2}
\begin{cases}
\frac{\dd}{\dd t}\pa_xu(t,\phi(t,x))\leq -\frac\gamma2\left(\pa_xu(t,\phi(t,x))+{C}_{\gamma}\|u_0\|_{H^1}\right)^2,\\
\frac{\dd}{\dd t}\pa_xu(t,\phi(t,x))\geq -\frac\gamma2\left(\pa_xu(t,\phi(t,x))-{C}_{\gamma}\|u_0\|_{H^1}\right)^2.
\end{cases}
\end{align}
Solving the above differential inequalities \eqref{ODE2} yields for  $x\in[-q_0+q^2_0,q_0-q^2_0]$
\bal\label{zw}
\frac{1}{\frac\gamma2t+\frac{1}{u'_0(0)-{C}_{\gamma}\|u_0\|_{H^1}}}+{C}_{\gamma}\|u_0\|_{H^1}\leq \pa_xu(t,\phi(t,x))\leq \frac{1}{\frac\gamma2t+\frac{1}{u'_0(0)+{C}_{\gamma}\|u_0\|_{H^1}}}-{C}_{\gamma}\|u_0\|_{H^1}.
\end{align}
From now on, for the sake of simplicity, we denote
\bbal
T^1_{\mathrm{max}}:=\frac{2}{\gamma}\frac{-1}{u'_0(0)+{C}_{\gamma}\|u_0\|_{H^1}}\quad\text{and}\quad
T^1_{\mathrm{min}}:=\frac{2}{\gamma}\frac{-1}{u'_0(0)-{C}_{\gamma}\|u_0\|_{H^1}}.
\end{align*}
We have proved that wave breaking for the rod equation \eqref{r} occurs at some time $T^*$, i.e., for $\gamma>0$, $\inf _{x \in \mathbb{R}}  \pa_xu(t, x)\to-\infty$ as $t \rightarrow T^*$. From \eqref{zw}, the solutions does not blow up
earlier than $T^1_{\mathrm{min}}$ and latter than $T^1_{\mathrm{max}}$, i.e., we can deduce that $T^*$ satisfies
$$T^1_{\mathrm{min}}\leq T^*\leq T^1_{\mathrm{max}}.$$
Notice that $u'_0(0)=-p_0e^{-q_0}$, one has
\bbal
T^1_{\mathrm{max}}=\frac2{\gamma p_0\f(e^{-q_0}-\tilde{C}_\gamma \frac{\|u_0\|_{H^1}}{p_0}\g)}\quad\text{and}\quad
T^1_{\mathrm{min}}=\frac2{\gamma p_0\f(e^{-q_0}+\tilde{C}_\gamma \frac{\|u_0\|_{H^1}}{p_0}\g)}.
\end{align*}
By Corollary \ref{co1}, one has $\frac{\|u_0\|_{H^1}}{p_0}\approx q_0^{\fr12}$. Then for $q_0$ small enough, we have
\bal\label{TT}
\frac{1}{\gamma p_0}\leq T^1_{\mathrm{min}}\leq T^*\leq T^1_{\mathrm{max}}\leq\frac3{\gamma p_0}.
\end{align}
Now we claim that, there exists an unique $\delta\in(0,q_0^2)$ such that
$u_0'(q_0+\delta)=u_0'(-q_0-\delta)=0$. 
Indeed, we need to verify that $u_0'(q_0+q_0^2)>0$ and $u_0'(q_0)<0$. By continuity and symmetry, we prove the above claim.
It is not difficult to obtain that $u_0'(q_0+q_0^2)>0$. Next, we divide the integral into parts
\bbal
u_0'(q_0)&=\int_{|q_0-y|<q_0^2}v'_0(y)J_{q_0^2}(q_0-y)\dd y\\
&=\f(\int_{\{y:-q_0^2<y-q_0\leq 0\}}+\int_{\{y:0<y-q_0<q_0^2\}}\g)v'_0(y)J_{q_0^2}(q_0-y)\dd y\\
&=:K_1+K_2.
\end{align*}
Then trivial computations give that
\bbal
&K_1=-p_0e^{-q_0}\int_{\{y:-q_0^2<y\leq0\}}J_{q_0^2}(y)\dd y=-\fr12p_0e^{-q_0},\\
&K_2=p_0q_0e^{-q_0}\int_{\{y:0<y< q_0^2\}}e^{-y}J_{q_0^2}(y)\dd y\leq\fr12p_0q_0e^{-q_0}.
 \end{align*}
Due to the fact that $q_0$ is small enough, then we obtain that $u_0'(q_0)<0$. 

{\bf Case 2: $x\in\mathbf{I}_2=(-\infty, -q_0-\delta]\cup[q_0+\delta,+\infty)$}.

Now, we claim that for some constant $M_{\gamma}>0$
 \bal\label{cla}
 |\pa_xu(t,\phi(t,x))|\leq M_{\gamma}p_0 q_0,\quad \forall (t,x)\in [0,T^*)\times \mathbf{I}_2.
 \end{align}
In fact, due to $0\leq u'_0(x)\leq 2p_0q_0$ if $x\in\mathbf{I}_2$, from \eqref{ODE1}, we have for $(t,x)\in [0,T^*)\times\mathbf{I}_2$
\bbal
&\pa_xu(t,\phi(t,x))\leq u'_0(x)+\frac{3+\gamma}{2}\|u_0\|_{H^1}^2T^*\leq M_{\gamma}p_0 q_0.
\end{align*}
From which and \eqref{ODE1}, we have for $(t,x)\in [0,T^*)\times\mathbf{I}_2$
\bbal
\pa_xu(t,\phi(t,x))\geq u'_0(x)-\f(\frac\gamma2\f(M_{\gamma}p_0 q_0\g)^2+\frac{3+\gamma}{2}\|u_0\|_{H^1}^2\g)T^*\geq -M_{\gamma}p_0 q_0.
\end{align*}
This gives \eqref{cla}, which tells us that the solution along the flow map do not blow up when $(t,x)\in [0,T^*)\times \mathbf{I}_2$.

{\bf Case 3: $x\in\mathbf{I}_3=(-q_0-\delta,-q_0+q_0^2)\cup (q_0-q_0^2,q_0+\delta)$}.

Since $u'_0(0)<u'_0(x)<0$, we consider the following two cases.
\begin{itemize}
  \item Suppose that $\pa_xu(t,\phi(t,x))>u'_0(0)$ holds for all $(t,x)\in [0,T^*)\times \mathbf{I}_3$, then we have $T^*\geq T^1_{\mathrm{min}}$.
  \item Suppose that $\pa_xu(t,\phi(t,x))>u'_0(0)$ holds for all $t\in[0,T_0)$ but $\pa_xu(T_0,\phi(T_0,x))=u'_0(0)$ for some $T_0\in(0,T^*)$, then we have for $t\in[T_0,T^*)$
\bbal
\frac{1}{\frac\gamma2(t-T_0)+\frac{1}{u'_0(0)-{C}_{\gamma}\|u_0\|_{H^1}}}+{C}_{\gamma}\|u_0\|_{H^1}&\leq \pa_xu(t,\phi(t,x))\\
&\leq \frac{1}{\frac\gamma2(t-T_0)+\frac{1}{u'_0(0)+{C}_{\gamma}\|u_0\|_{H^1}}}-{C}_{\gamma}\|u_0\|_{H^1},
\end{align*}
which implies
$$T_0+T^1_{\mathrm{min}}=T^2_{\mathrm{min}}\leq T^*\leq T^2_{\mathrm{max}}=T^1_{\mathrm{max}}+T_0.$$
\end{itemize}
Combining the above three cases, we deduce that $T^1_{\mathrm{min}}\leq T^*\leq T^1_{\mathrm{max}}$.

This completes the proof of Lemma \ref{life}.
\end{proof}

\subsection{Norm Inflation for Smooth Initial Data}\label{subsec44}

From \eqref{zw}, we deduce that for $(t,x)\in [0,T^*)\times[-q_0+q^2_0,q_0-q^2_0]$
\bal\label{pos}
m(t)+C_\gamma \|u_0\|_{H^1}\leq \pa_xu(t,\phi(t,x))\leq M(t)-C_\gamma \|u_0\|_{H^1},
\end{align}
where, for notational convenience, we denote
\bal\label{zw1}
m(t):=\frac{1}{\frac\gamma2t+\frac{1}{u'_0(0)-C_\gamma \|u_0\|_{H^1}}}\quad\text{and}\quad M(t):=\frac{1}{\frac\gamma 2t+\frac{1}{u'_0(0)+C_\gamma \|u_0\|_{H^1}}}.
\end{align}
From \eqref{pos} and \eqref{zw1}, and using Lemma \ref{le5}, we deduce that for $(t,x)\in [0,T^*)\times[\phi(t,q^2_0-q_0),\phi(t,q_0-q^2_0)]$
\bal\label{pos1}
m(t)+C_\gamma \|u_0\|_{H^1}\leq \pa_xu(t,x)\leq M(t)-C_\gamma \|u_0\|_{H^1}.
\end{align}
From \eqref{r}, then we obtain
\bbal
&\pa_t\pa_xu+\gamma u \pa_x^2u+\frac\gamma2\f(\pa_xu\g)^2=\frac{3-\gamma}{2}u^2-(1-\pa^2_x)^{-1}\f(\frac{3-\gamma}{2}u^2+\frac\gamma2\f(\pa_xu\g)^2\g)=:G,
\end{align*}
which implies that for $p>2$ which be fixed later
\bbal
\pa_t\f(\pa_xu\g)^p+\gamma  \pa_x\f(u\f(\pa_xu\g)^p\g)+\gamma\frac {p-2}2\f(\pa_xu\g)^{p+1}=pG\f(\pa_xu\g)^{p-1}.
\end{align*}
Integrating the above equation with respect to space variable $x$ over the interval $I=[\phi(t,q^2_0-q_0),\phi(t,q_0-q^2_0)]$, we have
\bal\label{p0}
\int_{I}\pa_t\f(\pa_xu\g)^p\dd x+\gamma\int_{I}\pa_x\f(u\f(\pa_xu\g)^p\g)\dd x+\gamma\frac {p-2}2\int_{I}\f(\pa_xu\g)^{p+1}\dd x
 =p\int_{I}G\f(\pa_xu\g)^{p-1}\dd x.
\end{align}
Direct computations yield that
\bal\label{p1}
\int_{I}\pa_t\f(\pa_xu\g)^p\dd x&=\frac{\dd}{\dd t}\int_{I}\f(\pa_xu\g)^p\dd x-\gamma u(t,\phi(t,q_0-q^2_0))\f(\pa_xu\g)^p(t,\phi(t,q_0-q^2_0)) \nonumber
\\&\quad +\gamma u(t,\phi(t,q^2_0-q_0))\f(\pa_xu\g)^p(t,\phi(t,q^2_0-q_0)),
\end{align}
and
\bal\label{p2}
\int_{I}\pa_x\f(u\f(\pa_xu\g)^p\g)\dd x &= u(t,\phi(t,q_0-q^2_0))\f(\pa_xu\g)^p(t,\phi(t,q_0-q^2_0))\nonumber
\\&\quad
- u(t,\phi(t,q^2_0-q_0))\f(\pa_xu\g)^p(t,\phi(t,q^2_0-q_0)).
\end{align}
Performing $\eqref{p1}+\gamma\times \eqref{p2}$, then we obtain from \eqref{p0} that
\bbal
&\quad \frac{\dd}{\dd t}\int_{I}\f(\pa_xu\g)^p\dd x+\gamma\frac {p-2}2\int_{I}\f(\pa_xu\g)^{p+1}\dd x
=p\int_{I}G\f(\pa_xu\g)^{p-1}\dd x.
\end{align*}
Letting
\bbal
A(p,t):=\int_{I}\f(-\pa_xu\g)^p(t,x)\dd x\quad\text{with}\quad A_0=\int_{|x|\leq q_0-q^2_0}\f(-u'_0(x)\g)^p\dd x,
\end{align*}
we have
\bbal
p\f|\int_{I}G\f(-\pa_xu\g)^{p-1}\dd x\g|\leq A(p,t)+C_{p,\gamma}\|u\|^2_{H^1},
\end{align*}
where we have used that for $p> 2$
$$\|G\|_{L^p} \leq C\|u\|_{H^1} \leq C\|u_0\|_{H^1} \leq Cp_0 q^{\frac12}_0.$$
Combining the above, we have
\bbal
A'(p,t)&=-\gamma\frac {p-2}2\int_{I}\pa_xu\f(-\pa_xu\g)^{p}(t,x)\dd x-p\int_{I}G\f(-\pa_xu\g)^{p-1}\dd x\\
&\geq \left(-\gamma\frac {p-2}2M(t)-1\right)A(p,t)-C p_0^2 q_0.
\end{align*}
Solving the above differential inequality and using Lemma \ref{life} give us that for $t\in [0,T^*)$
\bal\label{A0}
A(p,t)\geq&~ \exp\f\{B(p,t)-t\g\}\left(A_0-Cp_0^2 q_0t\right)\nonumber\\
\geq &~\frac23\exp\f\{B(p,t)\g\}\left(A_0-Cp_0^2 q_0t\right),
\end{align}
where
$$B(p,t)=\int^t_0 \left(-\gamma\frac {p-2}2M(\tau)\right)\dd\tau.$$
Easy computations give that$$B(p,t)=\int^t_0 \left(-\gamma\frac {p-2}2M(\tau)\right)\dd\tau=-(p-2)\ln \f(1+\fr\gamma2M(0)t\g),$$
then
\bbal
&\exp\f\{B(p,t)\g\}=\f(1+\fr\gamma2M(0)t\g)^{-(p-2)},
\end{align*}
and
\bbal
&A_0=\int_{|x|\leq q_0-q^2_0}\f(-u'_0(x)\g)^p\dd x\approx p_0^pq_0.
\end{align*}
Noticing that
$$T^1_{\mathrm{min}}=-\frac2\gamma\frac{1}{u'_0(0)-C_\gamma \|u_0\|_{H^1}},$$
then we have
\bal\label{A1}
\exp\f\{B(p,T^1_{\mathrm{min}})\g\}=&~\f(1+\fr\gamma2M(0)T^1_{\mathrm{min}}\g)^{-(p-2)}\nonumber\\
=&~\f(\frac{-u'_0(0)+C_\gamma \|u_0\|_{H^1}}{2C_\gamma \|u_0\|_{H^1}}\g)^{p-2}\nonumber\\
\geq&~ cq_0^{1-\frac{p}{2}}.
\end{align}
From \eqref{A0} and  \eqref{A1}, we have  for $p>2$
\bal\label{A2}
A(p,T^1_{\mathrm{min}})\geq \frac12\exp\f\{B(p,T_{\mathrm{min}})\g\}A_0\geq c p_0^{p}q_0^{2-\fr{p}2}.
\end{align}
From now on, taking
\begin{align}
\begin{cases}
p_s=\frac{1}{\frac32-s}\in(2,+\infty),\\
p_0=n\gg1,\\
q_0=\frac{1}{n^{p_s}\ln n}\ll1,
\end{cases}
\end{align}
then \eqref{A2} implies that
\bbal
A(p_s,T^1_{\mathrm{min}})\geq c p_0^{p_s}q_0^{2-\fr{p_s}2}\geq cn^{\frac{p_s(p_s-2)}{2}}\f(\ln n\g)^{\frac{p_s-4}{2}}.
\end{align*}
By Corollary \ref{co1}, one has for $s \in \f(1,\fr32\g)$
\bal\label{hl1}
\|u_0\|_{H^s}\leq  Cp_0q^{\frac32-s}_0\leq C\f(\ln n\g)^{s-\fr32},
\end{align}
but for $T_0\in[0,T^1_{\rm{min}})$
\bal\label{hl2}
\|u(T_{0})\|_{H^s}\geq c\|u(T_{0})\|_{{W}^{1,p_s}}
&\geq c\Big(n^{\frac{p_s-2}{2}}\f(\ln n\g)^{\frac{p_s-4}{2p_s}}\Big)^{\frac12},
\end{align}
where we have used the fact ${H}^s(\R)\hookrightarrow {W}^{1,p_s}(\R)$.

From \eqref{hl1} and \eqref{hl2}, we deduce that, there exists initial data $u^n_0$ with
\bbal
\|u^n_0\|_{H^s}\to0,\quad\text{as}\;n\to\infty,
\end{align*}
such that if we denote by $u^n(t)\in C^1([0,T^*);H^{\infty}(\mathbb{R}))$, the solution of \eqref{r} with initial data $u^n_0$ satisfies that
\bbal
\|u^n(t_n)\|_{H^s}\to\infty\quad\text{with}\; t_n\to 0,\quad\text{as}\;n\to\infty.
\end{align*}
In conclusion, we obtain the norm inflation and hence the ill-posedness of the rod equation \eqref{r}. Thus, Theorem \ref{th2} is proved.

\section*{Declarations}
\noindent\textbf{Data Availability}\\
No data was used for the research described in the article.

\vspace*{1em}
\noindent\textbf{Conflict of interest}\\
The authors declare that they have no conflict of interest.
\vspace*{1em}

\noindent\textbf{Funding}\\
J. Li is supported by the National Natural Science Foundation of China (12161004), Innovative High end Talent Project in Ganpo Talent Program (gpyc20240069), Training Program for Academic and Technical Leaders of Major Disciplines in Ganpo Juncai Support Program (20232BCJ23009), Jiangxi Provincial Natural Science Foundation (20252BAC210004). Y. Yu is supported by the National Natural Science Foundation of China (12101011).

\end{document}